\documentclass[12pt]{article}
\usepackage{amsmath,amscd,amsfonts,amssymb,latexsym,amsthm}
\usepackage{pgf}
\usepackage{authblk}
\usepackage[utf8]{inputenc}
\usepackage{graphicx}
\usepackage[OT4]{fontenc}
\usepackage{bbm}
\usepackage{booktabs}
\usepackage{stmaryrd}
\usepackage[normalem]{ulem}

\allowdisplaybreaks

\DeclareMathAlphabet{\mathmybb}{U}{bbold}{m}{n}

\newtheorem{lem}{Lemma}
\newtheorem{thm}{Theorem}
\newtheorem{deff}{Definition}
\newtheorem{cor}{Corollary}

\newcommand{\beq}{\begin{equation}}
\newcommand{\eeq}{\end{equation}}

\newcommand{\E}{\mathbb{E}}

\DeclareFontFamily{U}{mathx}{\hyphenchar\font45}
\DeclareFontShape{U}{mathx}{m}{n}{
	<5> <6> <7> <8> <9> <10>
	<10.95> <12> <14.4> <17.28> <20.74> <24.88>
	mathx10
}{}
\DeclareSymbolFont{mathx}{U}{mathx}{m}{n}
\DeclareMathSymbol{\bigtimes}{1}{mathx}{"91}

\newcommand{\ilsk}[2]{\left\langle #1\mid #2\right\rangle}

\newcommand{\prevfiltr}[1]{\mathcal{#1}_{\mathrm{prev}}}

\newcommand{\sgn}{\mathrm{sgn}\,}

\def\P{\mathbb{P}}

\def\Z{\mathbb{Z}}

\def\N{\mathbb{N}}

\def\sup{{\rm sup}}
\def\inf{{\rm inf}}

\renewcommand{\and}[0]{\ \&\ }

\title{One-sided Davis inequality for (F4) filtrations}
\author{Maciej Rzeszut\thanks{maciej.rzeszut@gmail.com}}
\affil{Military University of Technology, Warsaw}

\begin{document}
\maketitle

\abstract{The classical Davis inequality $\E Mf\simeq \E Sf$, where $(Sf)^2=\sum_{k}\left|f_{k}-f_{k-1}\right|^2$ is the square function and $Mf= \sup_n \left|f_n\right|$ is the maximal function, is true with a universal constant for any martingale $f$ on any filtration. A natural analog in the setting of (F4) doubly indexed filtrations, i.e. $\left(\mathcal{F}_{i,j}\right)_{i,j}$ such that the operators $\mathbb{E}\left(\cdot\mid \mathcal{F}_{i,\infty}\right)$ and $\E\left(\cdot\mid \mathcal{F}_{\infty,j}\right)$ commute and their product is $\E\left(\cdot\mid \mathcal{F}_{i,j}\right)$, is the conjecture
\[\mathbb{E}\sup_{n,m} \left|f_{n,m}\right|\simeq\mathbb{E}\left(\sum_{i,j}\left|\underbrace{f_{i,j}-f_{i-1,j}-f_{i,j-1}+f_{i-1,j-1}}_{\Delta f_{i,j}}\right|^2\right)^\frac{1}{2}.\]
It was known to be true only with some highly restrictive additional assumptions, e.g. regularity of the filtration ($g_{n,m}\gtrsim g_{n+1,m},g_{n,m+1}$ for any positive martingale $g$) or $f$ being a strong martingale ($\mathbb{E}\left(\Delta f_{i,j}\mid \mathcal{F}_{i-1,j}\vee \mathcal{F}_{i,j-1}\right)=0$). We prove the inequality $\lesssim $ assuming just the (F4) condition.}

\section{Introduction}
We will work with discrete-time filtrations. In a usual one-parameter setting, given an nondecreasing sequence $\mathcal{F}=\left(\mathcal{F}_k\right)_{\in\N}$ of $\sigma$-algebras such that $\{\emptyset,\Omega\}=\mathcal{F}_0\subseteq \mathcal{F}_1\subseteq \mathcal{F}_2\subseteq \ldots\subseteq \mathcal{F}_{\infty}:=\bigvee_{k=1}^\infty \mathcal{F}_k$, we will consider conditional expectation operators $\E_{\mathcal{F}_k}$ acting on $L^1(\Omega,\mathcal{F}_\infty,\P)$ and abbreviate to $\E_k$ when the filtration is implied. The martingale difference operators are given by 
\[\Delta_k =\begin{cases} \E_k-\E_{k-1} \text{ for }k\geq 1\\ \E_0=\E\text{ for }k=0\end{cases}\]
so that $f=\sum_{k=0}^\infty \Delta_k f$. We define the square function and the conditional square function, respectively by
\[S_\mathcal{F}f=\left(\sum_k\left|\Delta_k f\right|^2\right)^\frac{1}{2},\qquad s_\mathcal{F}f=\left(\sum_k\E_{k-1}\left|\Delta_k f\right|^2\right)^\frac{1}{2}.\]
Here, $\E_{-1}=\E_0$. Moreover, for any (not necessarily ordered) indexed family $\mathcal{G}=\left(\mathcal{G}_k\right)_{k\in J}$ of $\sigma$-algebras, $M_\mathcal{G}$ will denote its maximal function, i.e. 
\[M_\mathcal{G}f=\sup_k \left|\E_{\mathcal{G}_k}f\right|.\]
In terms of the three sublinear operators above, we define the respective Hardy spaces by the norms
\[\|f\|_{H^1_S[\mathcal{F}]}=\|S_\mathcal{F}f\|_{L^1},\quad \|f\|_{H^1_s[\mathcal{F}]}=\|s_\mathcal{F}f\|_{L^1},\quad \|f\|_{H^1_M[\mathcal{F}]}=\|M_\mathcal{F}f\|_{L^1}\]
and omit the filtration whenever implied. \par
The following is a classical result.
\begin{thm}[Davis inequality]
For any $f\in L^1$ and a one-parameter filtration $\mathcal{F}$,
\[\|f\|_{H^1_M[\mathcal{F}]}\simeq \|f\|_{H^1_S[\mathcal{F}]}.\]
\end{thm}
In particular, if a martingale admits an $L^1$ square function, then by the integrability of its maximal function, it is in fact closed (by an $L^1$ function). \par
Now take a doubly indexed sequence $\left(\mathcal{F}_{i,j}\right)_{i,j=0}^\infty$ such that $\mathcal{F}_{i,j}\subseteq \mathcal{F}_{i+1,j}\cap \mathcal{F}_{i,j+1}$. We will denote 
\[\mathcal{F}_{i,\infty}=\bigvee_{j=0}^\infty \mathcal{F}_{i,j},\qquad \mathcal{F}_{\infty,j}=\bigvee_{i=0}^\infty \mathcal{F}_{i,j}.\]
A usual assumption put on $\mathcal{F}_{i,j}$ is the so-called (F4) condition
\[\E_{i,j}\E_{k,l}=\E_{\min(i,k),\min(j,l)}=\E_{k,l}\E_{i,j}.\]
It is easily seen to be equivalent to the conditional independence of $\mathcal{F}_{i,\infty}$ and $\mathcal{F}_{\infty,j}$ given $\mathcal{F}_{i,j}$. For even more notational confusion, we will denote the entirety of the filtration $\left(\mathcal{F}_{i,\infty}\right)_{i=0}^\infty$ by $\mathcal{F}^{(1)}$ and $\left(\mathcal{F}_{\infty,j}\right)_{j=0}^\infty$ by $\mathcal{F}^{(2)}$. For a filtration $\mathcal{G}$ indexed by $\N$, by $\mathcal{G}_{\mathrm{prev}}$ we will denote $\left(\mathcal{G}_{\max(k-1,0)}\right)_{k\in\N}$, and if $\mathcal{G}$ is indexed by $\N^2$, then $\prevfiltr{G}=\left(\mathcal{G}_{\max(n-1,0),\max(m-1,0)}\right)_{n,m\in \N^2}$. \par
A natural analog of martingale difference operators are
\[\Delta_{i,j}= \left(\E_{i,\infty}-\E_{i-1,\infty}\right)\left(\E_{\infty,j}-\E_{\infty,j-1}\right)=\E_{i,j}-\E_{i-1,j}-\E_{i,j-1}+\E_{i-1,j-1}.\]
In particular, $\Delta_{i,j}= \Delta_{i,\infty}\Delta_{\infty,j}$ is the composition of martingale difference opeators with repsect to $\mathcal{F}^{(1)}$ and $\mathcal{F}^{(2)}$. The square functions and the respective Hardy spaces are defined in the same way as in the one-parameter case:
\[Sf=\left(\sum_{i,j}\left|\Delta_{i,j}f\right|^2\right)^\frac{1}{2},\qquad sf=\left(\sum_{i,j}\E_{i-1,j-1}\left|\Delta_{i,j}f\right|^2\right)^\frac{1}{2}.\]
It is known (see e.g. \cite{weiszbook} or \cite{brossard}) that 
\[ \|f\|_{H^1_s}\gtrsim \|f\|_{H^1_S},\|f\|_{H^1_M}.\]
However, the inequality $\|f\|_{H^1_S}\simeq \|f\|_{H^1_M}$ was known to be true only with some additional assumptions: either for both $\mathcal{F}^{(1)}, \mathcal{F}^{(2)}$ being regular, or $f$ being a strong martingale, i.e. $\E_{\mathcal{F}_{i-1,j}\vee \mathcal{F}_{i,j-1}}\Delta_{i,j} f=0$. \par
Our goal is to prove the inequality 
\[\|f\|_{H^1_S}\gtrsim\|f\|_{H^1_M}\]
without any additional assumptions. The proof is heavily inspired by Dilworth \cite{dilworth}. 

\section{Preliminaries}

\subsection{Banach lattices}
\begin{deff}
Let $X$ be a (quasi-)Banach lattice. The $\alpha$-convexification of $X$ is the space $X^{(\alpha)}$ of functions $f$ such that 
\[ \|f\|_{X^{(\alpha)}}:= \left\| |f|^\alpha \right\|^{1/\alpha} \]
is finite.
\end{deff}

\begin{deff}
Let $X,Y$ be (quasi-)Banach lattices of functions such that the algebraic sum $X+Y$ makes sense. The interpolation sum $X+Y$ is defined by the norm
\[\|f\|_{X+Y}=\inf_{\substack{ f=a+b \\ a\in X,b\in Y }}\|a\|_X+\|b\|_Y,\]
while the intersection $X\cap Y$ is defined by the norm
\[\|f\|_{X\cap Y}=\|f\|_X+\|f\|_Y.\]
\end{deff}
One can easily check that \[ (X+Y)^{(\alpha)}=X^{(\alpha)}+Y^{(\alpha)},\qquad (X\cap Y)^{(\alpha)}=X^{(\alpha)}\cap Y^{(\alpha)}\]
with constants in norm equivalences depending on $\alpha$. Moreover, if $X$ and $Y$ are Banach and $X\cap Y$ is dense in both of them, then \[(X+Y)^*=X^*\cap Y^*, \qquad (X\cap Y)^*= X^*+Y^*.\]

\subsection{Conditional independence}
Recal that we call two $\sigma$-algebras $\mathcal{F},\mathcal{G}$ \emph{conditionally independent} (given $\mathcal{F}\wedge\mathcal{G}$) if 
\[\E_\mathcal{F}\E_{\mathcal{G}}= \E_{\mathcal{F}\wedge\mathcal{G}}= \E_\mathcal{G}\E_{\mathcal{F}}.\]
This is easily understood if $\mathcal{F}\wedge\mathcal{G}$ is purely atomic, i.e. $\Omega=\dot\bigcup_{k\in K}A_k$ and $\mathcal{F}\wedge\mathcal{G}$ is generated by $A_k$'s. In this case, $\mathcal{F},\mathcal{G}$ are independent if restricted to each of the $A_k$'s as a new probability space. 
\begin{lem}
	For any $\sigma$-algebras $\mathcal{F},\mathcal{G}$ on the same probability space, the linear span of the set $\left\{\mathbbm{1}_{A}\cdot\mathbbm{1}_{B}:A\in\mathcal{F},B\in\mathcal{G}\right\}$ is dense in $L^2\left(\mathcal{F}\vee \mathcal{G}\right)$.
\end{lem}
\begin{proof}
	Let 
	\[R= \overline{\mathrm{span}}\left\{\mathbbm{1}_{A}\cdot\mathbbm{1}_{B}:A\in\mathcal{F},B\in\mathcal{G}\right\},\]
	\[\mathcal{P}=\left\{A\cap B: A\in\mathcal{F},B\in\mathcal{G}\right\},\]
	\[\mathcal{Q}=\left\{A\in \mathcal{F}\vee \mathcal{G}: \mathbbm{1}_{A}\in R\right\}.\]
	It is easily seen that $\mathcal{P}$ is a $\pi$-system while $\mathcal{Q}$ is a $\lambda$-system. Therefore, $\mathcal{Q}\supseteq \sigma(\mathcal{P})= \mathcal{F}\vee\mathcal{G}$, so $R$ is a closed subspace of $L^2\left(\mathcal{F}\vee \mathcal{G}\right)$ that contains all the measurable indicator functions, so $R=L^2\left(\mathcal{F}\vee \mathcal{G}\right)$. 
\end{proof}

\begin{lem}\label{fveeg}
Let $\mathcal{F},\mathcal{G}$ be conditionally independent. Then for any $X$,
\[\E_\mathcal{F} \left(\E_\mathcal{G} X^2\right)^\frac{1}{2}\geq \left(\E_\mathcal{G}\left(\E_\mathcal{F} X\right)^2\right)^\frac{1}{2}.\]
\end{lem}
If $\mathcal{F},\mathcal{G}$ are actually independent or at least $\mathcal{F}\wedge \mathcal{G}$ is purely atomic, the lemma is obvious. 
\begin{proof}
Since
\[\E_\mathcal{F} \left(\E_\mathcal{G} X^2\right)^\frac{1}{2} = \E_\mathcal{F} \E_{\mathcal{G}}\left(\E_\mathcal{G} X^2\right)^\frac{1}{2}= \E_{\mathcal{F}\wedge \mathcal{G}} \left(\E_\mathcal{G} X^2\right)^\frac{1}{2},\]
the LHS is $\mathcal{F}\wedge \mathcal{G}$-measurable and for the same reason, the RHS is. 
We will prove the inequality first for $X\in \mathrm{span}\left\{\mathbbm{1}_{A}\cdot\mathbbm{1}_{B}:A\in\mathcal{F},B\in\mathcal{G}\right\}$. Let 
\[X=\sum_{k\in K}x_{k}\mathbbm{1}_{A_k}\mathbbm{1}_{B_k}\]
for some finite set $K$ and $A_k\in \mathcal{F},B_k\in \mathcal{G}$. Then
\[X=\sum_{k\in K} \mathbbm{1}_{A_k} X_k\]
for some $X_k\in L^2(\mathcal{G})$. By replacing the family $\left(A_k\right)_{k\in K}$ with \[\left(\bigcap_{k\in L}A_k\cap\bigcap_{k\notin L} A_k'\right)_{L\subseteq K},\] we may WLOG assume that $A_k$'s are disjoint, so that $\mathbbm{1}_{A_k}\mathbbm{1}_{A_\ell}=0$ for $k\neq\ell$. Now
\begin{align*}\left(\E_\mathcal{F}X\right)^2=&\left( \sum_k \mathbbm{1}_{A_k}\E_\mathcal{F}X_k\right)^2\\
	=& \sum_{k,\ell} \mathbbm{1}_{A_k}\mathbbm{1}_{A_\ell}\E_\mathcal{F}X_k \E_\mathcal{F} X_\ell\\
	=& \sum_{k} \mathbbm{1}_{A_k}\left(\E_\mathcal{F}X_k\right)^2.
\end{align*}
By the $\mathcal{G}$-measurability of $\E_\mathcal{F}X_k$,
\[\left(\E_\mathcal{G}\left(\E_\mathcal{F}X\right)^2\right)^\frac{1}{2}= \left(\sum_k \E_\mathcal{G}\mathbbm{1}_{A_k} \left(\E_\mathcal{F}X_k\right)^2\right)^\frac{1}{2}.\]
On the other hand,
\begin{align*}\E_\mathcal{F} \left(\E_\mathcal{G} X^2\right)^\frac{1}{2}=& \E_\mathcal{F} \left(\E_\mathcal{G} \sum_{k,\ell}\mathbbm{1}_{A_k}\mathbbm{1}_{A\ell}X_kX_\ell\right)^\frac{1}{2}\\
	=& \E_\mathcal{F} \left(\E_\mathcal{G} \sum_{k}\mathbbm{1}_{A_k}X_k^2\right)^\frac{1}{2}\\
	=& \E_\mathcal{F} \left( \sum_{k}\E_\mathcal{G}\mathbbm{1}_{A_k}X_k^2\right)^\frac{1}{2}.
\end{align*}
Ultimately, writing the $\mathcal{F}\wedge\mathcal{G}$-measurable random variables $\E_\mathcal{G}\mathbbm{1}_{A_k}$ as $Y_k^2$,
\begin{align*}\E_\mathcal{F} \left(\E_\mathcal{G} X^2\right)^\frac{1}{2}=& \E_\mathcal{F} \left( \sum_{k}(Y_kX_k)^2\right)^\frac{1}{2}\\
	\geq &  \left( \sum_{k}(\E_\mathcal{F} Y_kX_k)^2\right)^\frac{1}{2}\\
	=& \left( \sum_{k}(Y_k \E_\mathcal{F} X_k)^2\right)^\frac{1}{2}\\
	=& \left(\E_\mathcal{G}\left(\E_\mathcal{F}X\right)^2\right)^\frac{1}{2}.
\end{align*}
Now fix a nonnegative $Z\in L^2\left(\mathcal{F}\wedge \mathcal{G}\right)$. We already know that 
\[ \E\left(Z\E_\mathcal{F} \left(\E_\mathcal{G} X^2\right)^\frac{1}{2}\right)\geq \E\left(Z\left(\E_\mathcal{G}\left(\E_\mathcal{F}X\right)^2\right)^\frac{1}{2}\right)\]
holds for $X$ in the linear span of $\left\{\mathbbm{1}_{A}\cdot\mathbbm{1}_{B}:A\in\mathcal{F},B\in\mathcal{G}\right\}$. Since both sides are continuous in the $L^2$ norm, it holds for $X\in L^2(\mathcal{F}\vee \mathcal{G})$. Ultimately, by the fact that $Z$ was arbitrary, we conclude that 
\[\E_\mathcal{F} \left(\E_\mathcal{G} X^2\right)^\frac{1}{2}\geq \left(\E_\mathcal{G}\left(\E_\mathcal{F} X\right)^2\right)^\frac{1}{2}.\]
\end{proof}

\begin{deff}A (not necessarily ordered) family $\left(\mathcal{U}_k\right)_{k\in J}$ of $\sigma$-algebras is said to have the Doob property if it satisfies the inequality
	\[\left\|M_{\mathcal{U}}X\right\|_{L^p}\lesssim_p \|X\|_{L^p}\]
	for any $p\in (1,\infty)$.
\end{deff}

\begin{lem}
An (F4)	filtration has the Doob property.
\end{lem}

\begin{proof}
By convexity of the supremum and the one-parameter Doob inequality,
\begin{align*}
\left\|\sup_{n,m}\E_{n,m}|X|\right\|_{L^p}=& \left\|\sup_{n}\sup_m \E_{n,\infty}\E_{\infty,m}|X|\right\|_{L^p}\\
\leq & \left\|\sup_{n} \E_{n,\infty} \sup_m \E_{\infty,m}|X|\right\|_{L^p}\\
\lesssim_p & \left\|\sup_m \E_{\infty,m}|X|\right\|_{L^p}\\
\lesssim_p &\|X\|_{L^p}.
\end{align*}	
\end{proof}

\subsection{Martingale inequalities}

\begin{thm}[Burkholder-Rosenthal inequality]
	For any $p\geq 2$ and suitable $f$,
	\[\|f\|_{L^p}\simeq \left(\E\sum_k \left|\Delta_k f\right|^p\right)^\frac{1}{p} + \left(\E\left(\sum_k \E_{k-1}\left|\Delta_k f\right|^2\right)^\frac{p}{2}\right)^\frac{1}{p}.\]
\end{thm}

\begin{cor}\label{burk-ros-ad}
	For any $q\geq 1$ and nonnegative $Y_k\in L^q$ adapted to $\mathcal{F}_k$,
	\[\E\left(\sum_k Y_k\right)^q\simeq \E\sum_k Y_k^q + \E\left(\sum_k \E_{k-1} Y_k\right)^q.\]
\end{cor}
\begin{proof}
	If $Y_k$ are defined on $\Omega$, define $f$ on $\Omega\otimes \Z_2^\N$ by $f=\sum_k Y_k^\frac12 r_k$, where $r_k$ are Rademacher variables independent of $Y_k$'s. Then $f$ may be thought of as a martingale with respect to the filtration $\left(\mathcal{F}_k\otimes \sigma\left(r_1,\ldots,r_k\right) \right)_{k=1}^\infty$. In particular, $\Delta_k f=Y_k^\frac12 r_k$, so $\left|\Delta_k f\right|= Y_k^\frac12$. By Burkholder-Gundy,
	\[\begin{aligned} \E\left(\sum_k Y_k\right)^{p/2}  =&\E\left(\sum_k \left|Y_k^\frac12 r_k\right|^2 \right)^{p/2}\\
		=&  \E\left(\sum_k \left|\Delta_k f\right|^2 \right)^{p/2}\\
		\simeq & \|f\|_{L^p}^p\\
		\simeq & \E\sum_k \left|\Delta_k f\right|^p + \E\left(\sum_k \E_{k-1}\left|\Delta_k f\right|^2\right)^{p/2} \\
		= & \E\sum_k Y_k^{p/2} + \E\left(\sum_k \E_{k-1}Y_k\right)^{p/2}.
	\end{aligned}\]
	By taking $p:=2q$ we recover the desired inequality.
\end{proof}
\begin{cor}\label{iter-burk-ros}
	Let $q\geq 1$ and $Y_{i,j}$ be adapted to $\mathcal{F}_{i,j}$. Then
	\[\begin{aligned} \E\left(\sum_{i,j}Y_{i,j}\right)^q\simeq & \sum_{i,j}\E Y_{i,j}^q\\
		+& \E\left(\sum_{i,j}\E_{i-1,j-1}Y_{i,j}\right)^q\\
		+& \sum_i\E\left(\sum_j \E_{\infty,j-1}Y_{i,j}\right)^q\\
		+& \sum_j\E\left(\sum_i \E_{i-1,\infty}Y_{i,j}\right)^q.\end{aligned}\]
\end{cor}
\begin{proof}
	Firstly, we apply Corollary \ref{burk-ros-ad} to $\left(\sum_j Y_{i,j}\right)_i$ adapted to $\mathcal{F}_{i,\infty}$ and proceed with each summand in a similar fashion. 
\end{proof}

\subsection{Duality results}

\begin{deff}
For \emph{any} two families $\mathcal{U}=\left(\mathcal{U}_k\right)_{k\in J}$, $\mathcal{V}=\left(\mathcal{V}_k\right)_{k\in J}$ of $\sigma$-algebras and $p,q\in[1,\infty]$, we define a norm on $\mathcal{V}$-adapted sequences of random variables by
\[\left(X_k\right)\mapsto \left(\E\left( \sum_k \E_{\mathcal{U}_k} \left|X_k\right|^p \right)^{q/p}\right)^{1/q}=:\|X\|_{L^q\left(\mathcal{V},\ell^p\mid \mathcal{U}\right)}.\]
\end{deff}

\begin{center}
	\parbox[][][]{12cm}{\textsf{\footnotesize{It may also be thought of as 
				\[ \|X\|_{L^q\left(\mathcal{V},\ell^p\mid \mathcal{U}\right)}= \left\| \left\| \left(\left(\E_{\mathcal{U}_k}\left|X_k\right|^p\right)^{1/p}\right)_{k\in J}\right\|_{\ell^p}\right\|_{L^q}.\]
				In particular, in the case of $\mathcal{V}$ being a product filtration on $\Omega=\Xi^\N$ for some probability space $\Xi$ and $\mathcal{U}=\prevfiltr{V}$, the operator $\E_{\mathcal{U}_k}$ applied to $|X_k|$ simply integrates away the dependence on the $k$-th variable, and then the space $L^q\left(\mathcal{V},\ell^p\mid \mathcal{U}\right)$ becomes just the subspace of $L^q\left(\Omega,\ell^p\left(\N,L^p(\Xi)\right)\right)$ consisting of $\prevfiltr{V}$-adapted sequences of $L^p(\Xi)$-valued random variables.}}}
\end{center}

The inequalities in the following lemma are far from sharp, but they will allow us to infer that the spaces we are dealing with are reflexive. 

\begin{lem}
Suppose $\mathcal{U},\mathcal{V}$ have the Doob property and $p,q\in(1,\infty)$. Then:\\
(i) if $p\geq q$, then $\|X\|_{L^q\left(\mathcal{V},\ell^p\mid \mathcal{U}\right)}\leq \|X\|_{L^p\left(\ell^p\right)}=\left(\sum_k \E |X_k|^p\right)^\frac{1}{p}$;\\
(ii) if $p<q$, then $\|X\|_{L^q\left(\mathcal{V},\ell^p\mid \mathcal{U}\right)}\lesssim_{p,q} \|X\|_{L^q\left(\ell^p\right)}= \left(\E\left(\sum_k |X_k|^p\right)^{\frac{q}{p}}\right)^\frac{1}{q}$;\\
(iii) any linear functional on $L^q\left(\mathcal{V},\ell^p\mid \mathcal{U}\right)$ is represented by a $\mathcal{V}$-adapted sequence of random variables.
\end{lem}

\begin{proof} 
The part (i) is trivial by Jensen. In order to prove (ii), let $r=\frac{q}{p}$. By the usual H\"older duality,
\begin{align*}
	\|X\|_{L^q\left(\mathcal{V},\ell^p\mid \mathcal{U}\right)}^p= & \left(\E\left(\sum_k \E_{\mathcal{U}_k}|X_k|^p\right)^r\right)^\frac{1}{r}\\
	=& 	\sup_{\|Z\|_{L^{r'}}\leq 1} \E\left(Z \sum_k \E_{\mathcal{U}_k}|X_k|^p\right)\\
	=& \sup_{\|Z\|_{L^{r'}}\leq 1} \sum_k \E\left(|X_k|^p \E_{\mathcal{U}_k}Z\right)\\
	\leq & \sup_{\|Z\|_{L^{r'}}\leq 1} \sum_k \E\left(|X_k|^p M_\mathcal{U} Z\right)\\
	=& \sup_{\|Z\|_{L^{r'}}\leq 1} \E\left(M_\mathcal{U} Z \sum_k |X_k|^p \right)\\
	\leq & \sup_{\|Z\|_{L^{r'}}\leq 1} \left\|M_\mathcal{U} Z\right\|_{L^{r'}} \left\|\sum_k |X_k|^p \right\|_{L^r}\\
	\lesssim_r & \sup_{\|Z\|_{L^{r'}}\leq 1} \left\| Z\right\|_{L^{r'}} \left\|\sum_k |X_k|^p \right\|_{L^r}\\
	=& \left(\E\left(\sum_k |X_k|^p\right)^r\right)^\frac{1}{r}\\
	=& \left(\E\left(\sum_k |X_k|^p\right)^{\frac{q}{p}}\right)^\frac{p}{q}.
\end{align*}
For any $\alpha,\beta\in (1,\infty)$, functionals on $L^\alpha\left(\ell^\beta\right)$ are represented by sequences of random variables, and by Dilworth \cite{dilworth} the adapted sequences $L^\alpha\left(\mathcal{V},\ell^\beta\right)$ are a complemented subspace. By (i) and (ii), the space in question embeds in $L^\alpha\left(\mathcal{V},\ell^\beta\right)$ for $\{\alpha,\beta\}\in\{p,q\}$, so the result follows. 
\end{proof}

\begin{lem}\label{lazydual}\par
For $p,q\in  (1,\infty)$ and two families $\mathcal{U}_k, \mathcal{V}_k$ of $\sigma$-algebras such that $\mathcal{U}_k\subseteq \mathcal{V}_k$ and $\mathcal{U}_k$ has the Doob property, the spaces $L^{q'}\left(\mathcal{V},\ell^{p'}\mid \mathcal{U}\right)$ and $L^q\left(\mathcal{V},\ell^p\mid \mathcal{U}\right)$ are, up to norm equivalence, dual to each other with the usual coupling $\ilsk{X}{Y}=\sum_k \E X_k Y_k$.
\end{lem}

\begin{proof}
For $p=q$ there is nothing to prove, as $L^q\left(\mathcal{V},\ell^p\mid \mathcal{U}\right)$ becomes just the $\ell^p$-direct sum of $L^p$. Assume WLOG $p<q$. Let $X\in L^q\left(\mathcal{V},\ell^p\mid \mathcal{U}\right)=:V$. Then for any $Y\in L^{q'}\left(\mathcal{V},\ell^{p'}\mid \mathcal{U}\right)=:W$,
\begin{align*}\|X\|_V\cdot\|Y\|_W \geq & \E \left(\sum_k \E_{\mathcal{U}_k} \left|X_k\right|^p\right)^{1/p} \left(\sum_k \E_{\mathcal{U}_k} \left|Y_k\right|^{p'}\right)^{1/p'}\\
\geq & \E\sum_k \left(\E_{\mathcal{U}_k} \left|X_k\right|^p\right)^{1/p} \left(\E_{\mathcal{U}_k} \left|Y_k\right|^{p'}\right)^{1/p'}\\
\geq & \E\sum_k \E_{\mathcal{U}_k}X_kY_k\\
=& \sum_k \E X_k Y_k.
\end{align*}
The equality up to a constant is achieved for 
\[Y_k=\left|X_k\right|^{p-1} \E_{\mathcal{U}_{k}}s^{q-p}\sgn X_k,\] 
where $s=\left(\sum_j \E_{\mathcal{U}_{j}} \left|X_j\right|^p\right)^\frac{1}{p}$. Indeed, for this choice of $Y_k$,
\begin{align*}
	\sum_k \E X_k Y_k = & \sum_k \E \left(\left|X_k\right|^p \E_{\mathcal{U}_{k}} s^{q-p}\right)\\
	=& \sum_k \E \left(  s^{q-p} \E_{\mathcal{U}_{k}}\left|X_k\right|^p\right)\\
	=& \E \left(  s^{q-p} \sum_k \E_{\mathcal{U}_{k}}\left|X_k\right|^p\right)\\
	=& \E \left(  s^{q-p} \cdot s^p\right)\\
	=& \E s^q.
\end{align*}
Moreover applying the H\"older inequality with exponents $\alpha,\beta$ given by 
\[\frac{1}{\alpha}=\frac{q'p}{qp'}, \quad\frac{1}{\beta}= \frac{(q-p)q'}{q},\]
we get
\begin{align*}
	\|Y\|_W^{q'} =& \left\|\left(\sum_k \E_{\mathcal{U}_{k}} \left(\left|X_k\right|^{p-1} \E_{\mathcal{U}_{k}} s^{q-p}\right)^{p'}\right)^{\frac{1}{p'}}\right\|_{L^{q'}}^{q'}\\
=& \left\|\left(\sum_k \E_{\mathcal{U}_{k}}\left( \left|X_k\right|^{p} \left(\E_{\mathcal{U}_k} s^{q-p}\right)^{p'}\right)\right)^{\frac{1}{p'}}\right\|_{L^{q'}}^{q'}\\
=& \left\|\left(\sum_k \left(\E_{\mathcal{U}_k} s^{q-p}\right)^{p'}\E_{\mathcal{U}_{k}} \left|X_k\right|^{p} \right)^{\frac{1}{p'}}\right\|_{L^{q'}}^{q'}\\
	\leq & \left\|\left(\sum_k \left(M_{\mathcal{U}} s^{q-p}\right)^{p'} \E_{\mathcal{U}_{k}} \left|X_k\right|^{p} \right)^{\frac{1}{p'}}\right\|_{L^{q'}}^{q'}\\
	=& \left\|s^{\frac{p}{p'}} M_{\mathcal{U}} s^{q-p}\right\|_{L^{q'}}^{q'}\\
	=& \E s^{\frac{pq'}{p'}} \left(M_{\mathcal{U}} s^{q-p}\right)^{q'}\\
	\leq & \left(\E s^{\frac{pq'}{p'}\alpha}\right)^\frac{1}{\alpha} \left(\E \left(M_{\mathcal{U}} s^{q-p}\right)^{q'\beta}\right)^\frac{1}{\beta}\\
	=& \left(\E s^q\right)^\frac{1}{\alpha} \left(\E \left(M_{\mathcal{U}} s^{q-p}\right)^{\frac{q}{q-p}}\right)^\frac{1}{\beta}\\
	\lesssim & \left(\E s^q\right)^\frac{1}{\alpha} \left(\E \left(s^{q-p}\right)^{\frac{q}{q-p}}\right)^\frac{1}{\beta}\\
	= & \E s^q.
\end{align*}
Thus,
\begin{align*}
	\ilsk{X}{Y} =& \E s^q\\
	=& \left(\E s^q\right)^\frac{1}{q} \left(\E s^q\right)^{1-\frac{1}{q}}\\
	=& \|X\| \left(\E s^q\right)^{\frac{1}{q'}}\\
	\gtrsim & \|X\|_V \|Y\|_W.	
\end{align*}
Since any functional on $W=L^{q'}\left(\mathcal{V},\ell^{p'}\mid \mathcal{U}\right)$ is represented by a $\mathcal{V}$-adapted sequence of random variables, then the norm of $X$ in the dual to $W$ is $$\sup_{\|Y\|_W\leq 1}\ilsk{X}{Y}\simeq \|X\|_{V},$$ thus proving $W^*=V$. But then, by Hahn-Banach Theorem, for any $Y$ in $W$ we have $\|Y\|_W \simeq \sup_{\|X\|_{V}\leq 1}\ilsk{X}{Y}$, which in turn proves that $V^*=W$. 
\end{proof}

\section{Proof of the main theorem}

\begin{thm}
For any adapted sequence $\left(X_{i,j}\right)_{i,j}$,
\begin{align*}\E\left(\sum_{i,j}X_{i,j}^2\right)^\frac{1}{2}\simeq \inf_{X=A+B+C+D}& \sum_{i,j}\E \left|A_{i,j}\right| \\
+& \E\left(\sum_{i,j} \E_{i-1,j-1}B_{i,j}^2\right)^\frac{1}{2}\\
+& \sum_i \E\left(\sum_j \E_{\infty,j-1}C_{i,j}^2\right)^\frac{1}{2}\\
+ &\sum_j \E\left(\sum_i \E_{i-1,\infty}D_{i,j}^2\right)^\frac{1}{2},\end{align*}
where the infimum is taken over all decomposition of $X_{i,j}$ into a sum of adapted sequences $A_{i,j},B_{i,j},C_{i,j},D_{i,j}$. 
\end{thm}
\begin{proof}
By convexification, the desired inequality is equivalent to 
\begin{align*}\left[\E\left(\sum_{i,j}X_{i,j}^4\right)^\frac{1}{2}\right]^\frac{1}{2} \simeq \inf_{X=A+B+C+D}& \left[\sum_{i,j}\E A_{i,j}^2\right]^\frac{1}{2} \\
+&\left[ \E\left(\sum_{i,j} \E_{i-1,j-1}B_{i,j}^4\right)^\frac{1}{2}\right]^\frac{1}{2}\\
+& \left[\sum_i \E\left(\sum_j \E_{\infty,j-1}C_{i,j}^4\right)^\frac{1}{2}\right]^\frac{1}{2}\\
+ &\left[\sum_j \E\left(\sum_i \E_{i-1,\infty}D_{i,j}^4\right)^\frac{1}{2}\right]^\frac{1}{2}.\end{align*}
This variant has the advantage of featuring norms that allow us to calculate duals explicitly. Indeed, the norm
\[ \left(A_{i,j}\right)\mapsto \left[\sum_{i,j}\E A_{i,j}^2\right]^\frac{1}{2}\]
is self-dual, as it is in fact a Hilbert space norm, namely the $\ell^2$-sum of $L^2\left(\mathcal{F}_{i,j}\right)$. The dual to 
\[ \left(B_{i,j}\right)_{i,j}\mapsto \left[\E\left(\sum_{i,j} \E_{i-1,j-1}B_{i,j}^4\right)^\frac{1}{2}\right]^\frac{1}{2} \]
is, by Lemma \ref{lazydual} applied to $\left(\mathcal{F}_{i-1,j-1},\mathcal{F}_{i,j}\right)$, the norm 
\[ \left(\Phi_{i,j}\right)_{i,j}\mapsto \left[\E\left(\sum_{i,j} \E_{i-1,j-1}\Phi_{i,j}^{\frac{4}{3}}\right)^\frac{3}{2}\right]^\frac{1}{2}.\]
The norm 
\[\left(C_{i,j}\right)_{i,j}\mapsto \left[\sum_i \E\left(\sum_j \E_{\infty,j-1}C_{i,j}^4\right)^\frac{1}{2}\right]^\frac{1}{2}\]
is an $\ell^2$ sum of one parameter norms
\[ \left(X_j\right)_j\mapsto \left[\E\left(\sum_j \E_{\infty,j-1}X_{j}^4\right)^\frac{1}{2}\right]^\frac{1}{2} \]
defined on $\mathcal{F}_{i,\infty}$-measurable $X_j$'s for each $i$ separately. Thus, its dual is the $\ell^2$ sum of their duals, which are known from the one-parameter case or follow from Lemma \ref{lazydual}:
\[\left(\Phi_{i,j}\right)_{i,j}\mapsto \left[\sum_i \E\left(\sum_j \E_{\infty,j-1}\Phi_{i,j}^\frac{4}{3}\right)^\frac{3}{2}\right]^\frac{1}{2}.\]
By swapping the roles of $i$ and $j$, we get the same for 
\[\left(D_{i,j}\right)_{i,j}\mapsto \left[\sum_j \E\left(\sum_i \E_{i-1,\infty}D_{i,j}^4\right)^\frac{1}{2}\right]^\frac{1}{2}.\]
Ultimately, we are left with the dual of 
\[\left(X_{i,j}\right)_{i,j}\mapsto \left[\E\left(\sum_{i,j}X_{i,j}^4\right)^\frac{1}{2}\right]^\frac{1}{2}.\]
By Lemma \ref{lazydual} for $\left(\mathcal{F}_{i,j},\mathcal{F}_{i,j}\right)$ it is 
\[\left(\Phi_{i,j}\right)_{i,j}\mapsto \left[\E\left(\sum_{i,j}\Phi_{i,j}^\frac{4}{3}\right)^\frac{3}{2}\right]^\frac{1}{2}.\]
Ultimately, the desired inequality is by duality equivalent to 
\begin{align*}
	\left[\E\left(\sum_{i,j}Y_{i,j}^\frac{4}{3}\right)^\frac{3}{2}\right]^\frac{1}{2} \simeq & \left[\sum_{i,j}\E Y_{i,j}^2\right]^\frac{1}{2} \\
	+&\left[ \E\left(\sum_{i,j} \E_{i-1,j-1}Y_{i,j}^\frac{4}{3}\right)^\frac{3}{2}\right]^\frac{1}{2}\\
	+& \left[\sum_i \E\left(\sum_j \E_{\infty,j-1}Y_{i,j}^\frac{4}{3}\right)^\frac{3}{2}\right]^\frac{1}{2}\\
	+ &\left[\sum_j \E\left(\sum_i \E_{i-1,\infty}Y_{i,j}^\frac{4}{3}\right)^\frac{3}{2}\right]^\frac{1}{2}.
\end{align*}
Squaring both sides and putting $Z_{i,j}=Y_{i,j}^\frac{4}{3}$ we arrive at 
\begin{align*}
	\E\left(\sum_{i,j}Z_{i,j}\right)^\frac{3}{2}\simeq & \sum_{i,j}\E Z_{i,j}^\frac{3}{2} \\
	+&\E\left(\sum_{i,j} \E_{i-1,j-1}Z_{i,j}\right)^\frac{3}{2}\\
	+&\sum_i \E\left(\sum_j  \E_{\infty,j-1}Z_{i,j}\right)^\frac{3}{2}\\
	+ &\sum_j \E\left(\sum_i \E_{i-1,\infty}Z_{i,j}\right)^\frac{3}{2}.
\end{align*}
This is true by Corollary \ref{iter-burk-ros} and we are done. 
\end{proof}
\begin{cor}
	For any $f$,
	\[\E S_{\mathcal{F}}\gtrsim \E M_{\mathcal{F}}f.\]
\end{cor}
\begin{proof}
	Apply the main theorem for $X_{i,j}=\Delta_{i,j}f$. This results in a decomposition \[\Delta_{i,j}f=A_{i,j}+B_{i,j}+C_{i,j}+D_{i,j}\] 
	into adapted sequences such that
	\begin{align*}
	\E\left(\sum_{i,j} \left|\Delta_{i,j} f\right|^2\right)^\frac{1}{2} \gtrsim & \sum_{i,j}\E \left|A_{i,j}\right| \\
	+& \E\left(\sum_{i,j} \E_{i-1,j-1}B_{i,j}^2\right)^\frac{1}{2}\\
	+& \sum_i \E\left(\sum_j \E_{\infty,j-1}C_{i,j}^2\right)^\frac{1}{2}\\
	+ &\sum_j \E\left(\sum_i \E_{i-1,\infty}D_{i,j}^2\right)^\frac{1}{2}.
	\end{align*}
	Applying $\Delta_{i,j}$ to both sides we get
	\[\Delta_{i,j}f=\Delta_{i,j}A_{i,j}+ \Delta_{i,j}B_{i,j}+ \Delta_{i,j}C_{i,j}+ \Delta_{i,j}D_{i,j}.\] 
We will now show that the operator $\left(\Phi_{i,j}\right)_{i,j}\mapsto \left(\Delta_{i,j}\Phi_{i,j}\right)_{i,j}$ is bounded in all four norms appearing in the inteprolation sum. Trivially,
	\[\sum_{i,j}\E \left|A_{i,j}\right|\gtrsim \sum_{i,j}\E \left|\Delta_{i,j}A_{i,j}\right|.\]
	For any $i,j$, by Jensen inequality $\E_{i-1,\infty} B_{i,j}^2\geq \E_{i-1,\infty}\left(\Delta_{i,\infty} B_{i,j}\right)^2$, so $\E_{i-1,j-1} B_{i,j}^2\geq \E_{i-1,j-1}\left(\Delta_{i,\infty} B_{i,j}\right)^2$ and by symmetry $\E_{i-1,j-1}\left(\Delta_{i,\infty} B_{i,j}\right)^2\geq \E_{i-1,j-1}\left(\Delta_{\infty,j}\Delta_{i,\infty} B_{i,j}\right)^2$. Thus
\[\E\left(\sum_{i,j} \E_{i-1,j-1}B_{i,j}^2\right)^\frac{1}{2}\geq \E\left(\sum_{i,j} \E_{i-1,j-1}\left(\Delta_{i,j}B_{i,j}\right)^2\right)^\frac{1}{2}.\]
Similarly
	\[\sum_i \E\left(\sum_j \E_{\infty,j-1}C_{i,j}^2\right)^\frac{1}{2}\geq \sum_i \E\left(\sum_j \E_{\infty,j-1}\left(\Delta_{\infty,j}C_{i,j}\right)^2\right)^\frac{1}{2}\]
and since for each $i$, by Lemma \ref{fveeg} for $\mathcal{F}_{i-1,\infty}$ and $\mathcal{F}_{\infty,j-1}$,
\begin{align*} \E\left(\sum_j \E_{\infty,j-1}C_{i,j}^2\right)^\frac{1}{2} 
= & \E \E_{i-1,\infty}\left(\sum_j \left[\left(\E_{\infty,j-1}C_{i,j}^2\right)^\frac{1}{2} \right]^2 \right)^\frac{1}{2} \\ 
\geq & \E\left(\sum_j \left[\E_{i-1,\infty} \left(\E_{\infty,j-1}C_{i,j}^2\right)^\frac{1}{2} \right]^2 \right)^\frac{1}{2}\\
\geq &  \E\left(\sum_j \left[\left(\E_{\infty,j-1}\left(\E_{i-1,\infty}C_{i,j}\right)^2\right)^\frac{1}{2} \right]^2 \right)^\frac{1}{2}\\
= & \E\left(\sum_j \E_{\infty,j-1}\left(\E_{i-1,\infty}C_{i,j}\right)^2\right)^\frac{1}{2},
\end{align*}
the operator $\left(\Phi_{i,j}\right)\mapsto \left(\Delta_{i,\infty}\Phi_{i,j}\right)$ is bounded in this norm as well. The same follows for by symmetry for the last one as well. \par
Therefore, putting $A=\sum_{i,j}\Delta_{i,j}A_{i,j}$ etc., we get 
\[f=A+B+C+D\]
and 
\begin{align*}
		\E\left(\sum_{i,j} \left|\Delta_{i,j} f\right|^2\right)^\frac{1}{2} \gtrsim & \sum_{i,j}\E \left|\Delta_{i,j}A\right| \\
	+& \E\left(\sum_{i,j} \E_{i-1,j-1}\left(\Delta_{i,j}B\right)^2\right)^\frac{1}{2}\\
	+& \sum_i \E\left(\sum_j \E_{\infty,j-1}\left(\Delta_{i,j}C\right)^2\right)^\frac{1}{2}\\
	+ &\sum_j \E\left(\sum_i \E_{i-1,\infty}\left(\Delta_{i,j}D\right)^2\right)^\frac{1}{2}.
\end{align*}
Since $\E_{n,m}\Delta_{i,j}= \Delta_{i,j}\mathbbm{1}_{i\leq n\text{ and }j\leq m}$, the inequality 
\[\sum_{i,j}\E \left|\Delta_{i,j}A\right|= \sum_{i,j}\E M_{\mathcal{F}}\Delta_{i,j}A= \sum_{i,j}\left\|\Delta_{i,j}A\right\|_{H^1_M}\geq \left\|\sum_{i,j}\Delta_{i,j}A\right\|_{H^1_M}=\|A\|_{H^1_M}\]
is straightforward. The inequality 
\[\E\left(\sum_{i,j} \E_{i-1,j-1}\left(\Delta_{i,j}B\right)^2\right)^\frac{1}{2}\gtrsim \|B\|_{H^1_M}\]
is just the Brossard theorem \cite{brossard}. Lastly, the mixed norms are bounded by the same result in the one-parameter version and the elementary inequality $\sup_i |x_i|\leq |x_0|+\sum_{i\geq 1} |x_i - x_{i-1}|$:
\begin{align*} \sum_i \E\left(\sum_j \E_{\infty,j-1}\left(\Delta_{i,j}C\right)^2\right)^\frac{1}{2}=& \sum_i \left\|\Delta_{i,\infty}C\right\|_{H^1_s\left[\mathcal{F}^{(2)}\right]}\\
\gtrsim & \sum_i \left\|\Delta_{i,\infty}C\right\|_{H^1_M\left[\mathcal{F}^{(2)}\right]}\\
=& \E\sum_i \sup_j \left|\E_{\infty,j}\Delta_{i,\infty}C\right|\\
\geq & \E\sup_j \sum_i \left|\E_{\infty,j}\Delta_{i,\infty}C\right|\\
\geq & \E\sup_j \sup_i \left|\E_{\infty,j}\E_{i,\infty}C\right|\\
= & \|C\|_{H^1_M}.
\end{align*}
Ultimately,
\[\|f\|_{H^1_S}\gtrsim \|A\|_{H^1_M}+\|B\|_{H^1_M}+\|C\|_{H^1_M}+\|D\|_{H^1_M}\geq \|A+B+C+D\|_{H^1_M}=\|f\|_{H^1_M}.\]
\end{proof}


\begin{thebibliography}{ams}
	\bibitem{brossard} J. Brossard, \emph{Regularite des martingales a deux indices et inegalites de normes}, in: Korezlioglu, H., Mazziotto, G., Szpirglas, J. (eds) Processus Aléatoires à Deux Indices. Lecture Notes in Mathematics, vol 863. Springer, Berlin, Heidelberg. (1981) https://doi.org/10.1007/BFb0091095
	\bibitem{dilworth} S. J. Dilworth, \emph{Some probabilistic inequalities with applications to functional analysis}, Banach Spaces (Bor-LuhLin, W.	B.Johnson ed.), Contemp. Math., AMS (1992)
	\bibitem{weiszbook} F. Weisz, \textit{Martingale Hardy Spaces and Their Applications in Fourier Analysis}, Lecture Notes in Mathematics 1568, Springer, Berlin, 1994.
\end{thebibliography}
\end{document}